\documentclass[notitlepage]{article}
\usepackage{graphicx, amsmath, amssymb,amsthm, geometry}
\usepackage{indentfirst} 

\newtheorem{Thm}{Theorem}

\newtheorem{prop}{Proposition}
\newtheorem{remark}{Remark}

\newtheorem{defn}{Definition}
\newtheorem{coro}{Corollary}

\def\*#1{\mathbf{#1}} 

\title{The existence of horizontal envelopes in the 3D-Heisenberg group}
\author{Yen-Chang Huang \\ Xinyang Normal University, \\Henan, China\\ ychuang@xynu.edu.cn}
\date{}

\begin{document}
\maketitle
\begin{abstract}
By using the support function on the $xy$-plane, we show the necessary and sufficient conditions for the existence of envelopes of horizontal lines in the 3D-Heisenberg group. A method to construct horizontal envelopes from the given ones is also derived, and we classify the solutions satisfying the construction.
\end{abstract}

\section{Introduction}
Given a family of lines in $\mathbb{R}^2$. The envelope of the family of lines $F(\lambda, x,y)$, depending on the parameter $\lambda$, is defined to be a curve on which every line in the family contacts exactly at one point. A simple example is to find the envelope of the family of lines $F(\lambda,x,y)=(1-\lambda)x+\lambda y -\lambda(1-\lambda)=0$ for $\lambda\in [0,1]$ and $(x,y)\in [0,1]\times [0,1]$. Consider the system of differential equations
\begin{align}\label{system1}
\left\{
\begin{array}{rl}
F(\lambda, x,y)&=0,\\
\frac{\partial F(\lambda,x,y)}{\partial\lambda}&=0.
\end{array}
\right.
\end{align}The second equation helps us find $\lambda$ in terms of $x,y$, and substitute $\lambda$ into the first equation to have the envelope
\begin{align}\label{envelope}
x^2 + y^2 - 2xy - 2x - 2y + 1=0.
\end{align}
Note that there are totally three variables $\lambda, x,y$, two equations in the system of differential equations \eqref{system1}, and hence one gets the solution \eqref{envelope} which is a one-dimensional curve in $\mathbb{R}^2$. One of applications in Economics to finding the envelopes is the Envelope Theorem which is related to gain the optimal production functions for given input and output prices \cite{CS, MS}. Recently the theorem has been generalized to the functions of multivariables with non-differentiability condition \cite{GA}. In general, it is impossible to seek the envelope of a family of lines in the higher dimensional Euclidean spaces $\mathbb{R}^n$ for $n\geq 3$ due to the over-determined system of differential equations similar to \eqref{system1} (Sec. 26, Chapter 2 \cite{E}). However, by considering the $3$-dimensional Heisenberg group $\mathbb{H}_1$ as a sub-Riemannian manifold, the horizontal envelopes can be obtained by a family of horizontal lines. In this notes we will study the problem of existence for horizontal envelopes in $\mathbb{H}_1$.

We recall some terminologies for our purpose. For more details about the Heisenberg groups, we refer the readers to \cite{CHL, M, MR, RR}. The $3$-dimensional Heisenberg group $\mathbb{H}_1$ is the Lie group $(\mathbb{R}^3, \star)$, where the group operation $\star$ is defined, for any point $(x,y,z)$, $(x',y',z')\in\mathbb{R}^3$, by
$$(x,y,z)\star (x',y',z')=(x+x',y+y',z+z'+yx'-xy').$$
For $p\in\mathbb{H}^1$, the \textit{left translation} by $p$ is the diffeomorphism $L_p(q):=p\star q$. A basis of left invariant vector fields (i.e., invariant by any left translation) is given by
$$\mathring{e}_1(p):=\frac{\partial}{\partial x}+y\frac{\partial}{\partial z}, \  \mathring{e}_2(p):=\frac{\partial}{\partial y}-x\frac{\partial}{\partial z}, \ T(p):=(0,0,1).$$
The \textit{horizontal distribution} (or \textit{contact plane} $\xi_p$ at any point $p\in \mathbb{H}^1)$ is the smooth planar one generated by $\mathring{e}_1(p)$ and $\mathring{e}_2(p)$. We shall consider on $\mathbb{H}_1$ the (left invariant) Riemannian $g:=\langle .,.\rangle$ so that $\{\mathring{e}_1,\mathring{e}_2,T\}$ is an orthonormal basis in the Lie algebra of $\mathbb{H}^1$. The endomorphism $J:\mathbb{H}_1\rightarrow \mathbb{H}_1$ is defined such that $J(\mathring{e}_1)=\mathring{e}_2$, $J(\mathring{e}_2)=-\mathring{e}_1$, $J(T)=0$ and $J^2=-1$.

A curve $\gamma: I \subset \mathbb{R}\rightarrow \mathbb{H}_1$ is called \textit{horizontal} (or \textit{Legendrian}) if its tangent at any point on the curve is on the contact plane. More precise, if we write the curve in the coordinates $\gamma:=(x,y,z)$ with the tangent vector $\gamma'=(x',y',z')=x'\mathring{e}_1(\gamma)+y'\mathring{e}_2(\gamma)+T(z'-x'y+xy')$, then the curve $\gamma$ is horizontal if and only if
\begin{align}\label{horizontal}
z'-x'y+xy'=0,
\end{align}where the prime $'$ denotes the derivative with respect to the parameter of the curve. The velocity $\gamma'$ has the natural decomposition
\begin{align*}
\gamma'=\gamma'_\xi+\gamma'_T,
\end{align*}
where $\gamma'_\xi$ (resp. $\gamma'_T$) is the orthogonal projection of $\gamma'$ on $\xi$ along $T$ (resp. on $T$ along $\xi$) with respect to the metric $g$. Recall that a \textit{horizontally regular curve} is a parametrized curve $\gamma(t)$ such that $\gamma'_\xi(u)\neq 0$ for all $u\in I $ (Definition 1.1 \cite{CHL}). Also, in Proposition 4.1 \cite{CHL}, we show that any horizontally regular curve can be uniquely reparametrized by \textit{horizontal arc-length} $s$, up to a constant, such that $|\gamma'_\xi(s)|=1$ for all $s$, and called the curve being \textit{with horizontal unit-speed}.
Moreover, two geometric quantities for horizontally regular curves parametrized by arc-length, the \textit{p-curvature} $k(s)$ and the \textit{contact normality} $\tau(s)$, are defined by
\begin{align*}
k(s)&:=\langle\frac{d\gamma'(s)}{ds},J\gamma'(s)\rangle, \\
\tau(s)&:=\langle\gamma'(s),T\rangle,
\end{align*}which are invariant under pseudo-hermitian transformations of horizontally regular curves (Section 4, \cite{CHL}). Note that $k(s)$ is analogous to the curvature of the curve in the Euclidean space $\mathbb{R}^3$, while $\tau(s)$ measures how far the curve is from being horizontal. When the curve $\gamma(u)$ is parametrized by arbitrary parameter $u$ (not necessarily its arc-length $s$), the p-curvature is given by
\begin{align}\label{pcurve}
k(u):=\frac{x'y''-x''y'}{\left((x')^2+(y')^2\right)^{3/2}}(u).
\end{align}

It is clear that a curve $\gamma (s)$ is a horizontal if and only if $\gamma(s)=\gamma_\xi(s)$ for all $s$. One of examples for horizonal curves is the horizontal lines which can be characterized by the following proposition.

\begin{prop}\label{prop1} Any horizontal line $\ell$ in $\mathbb{H}_1$ can be uniquely determined by three parameters $(p,\theta, t)$ for any $p\geq 0$, $\theta\in[0, 2\pi]$, $t\in \mathbb{R}$. Any point on the line can be represented in the coordinates
\begin{align}\label{projpeq3}
\ell: \left\{
\begin{array}{rl}
x&=p\cos\theta - s\ \sin\theta, \\
y&=p\sin\theta + s\ \cos\theta, \\
z&=t         -sp,
\end{array}\right.
\end{align}for all $s\in \mathbb{R}$.
\end{prop}

When $s=0$, denote by
\begin{align}\label{Q}
Q':=(x,y,z)=(p\cos\theta, p\sin\theta, t).
\end{align}
Observe that $\ell$ is a line through the point $Q'$ with the directional vector $\sin\theta\ \mathring{e}(Q')+\cos\theta\ \mathring{e}(Q')$. The value $p$ actually is the value of the support function $p(\theta)$ for the projection $\pi(\ell)$ of $\ell$ on the $xy$-plane (see the proof of Proposition \ref{prop1} in next section).

Inspired by the envelopes in the plane, with the assistance of contact planes in $\mathbb{H}_1$ we introduce the horizontal envelope tangent to a family of horizontal lines.
\begin{defn}
Given a family of horizontal lines in $\mathbb{H}_1$. A \textbf{horizontal envelope} is a horizontal curve $\gamma$ such that $\gamma$ contacts with exactly one line in the family at one point.
\end{defn}

Back to Proposition \ref{prop1}, $p$ actually is the distance of the projection $\pi(\ell)$ of $\ell$ onto the $xy$-plane to the origin, and $\theta$ is the angle from the $x$-axis to the line perpendicular to the projection (see Fig. \ref{envelope1} next section). To obtain the horizontal envelope $\gamma$, it is natural to consider $p$ as a function of $\theta$, namely, the support function $p=p(\theta)$ for the projection $\pi(\gamma)$ of the curve on the $xy$-plane. Moreover, by \eqref{Q} we know that the value $t$ dominates the height of the point $Q'$. As long as $\theta$ is fixed, the projection of $Q'$ onto the $xy$-plane is fixed. Thus, we may consider $t$ as a function of $\theta$. Under these circumstances, the family of horizontal lines is only controlled by one parameter $\theta$, and so the following is our main theorem.

\begin{Thm}\label{mainthm1}
Let $p=p(\theta)\geq 0$ and $t=t(\theta)$ be $C^1$-functions defined on $\theta\in[0,2\pi]$ satisfying
\begin{align}\label{thmcondi}
t'=(p')^2-p^2.
\end{align} There exists a horizontally regular curve $\gamma$ parametrized by arc-length such that the curve is the horizontal envelope of the family of horizontal lines determined by $\theta$ (and hence $p$ and $t$). In the coordinates, the envelope $\gamma=(x,y,z)$ can be represented by
\begin{align}\label{maineq1}
\gamma:
\left\{
\begin{array}{rl}
x(\theta)&= p(\theta)\cos\theta -p'(\theta) \sin\theta, \\
y(\theta)&= p(\theta)\sin\theta +p'(\theta) \cos\theta, \\
z(\theta)&=t(\theta)-p'(\theta) p(\theta).
\end{array}
\right.
\end{align}
Moreover, if $p$ is a $C^2$-function, the $p$-curvature and the contact normality of $\gamma$ are given by $k=\frac{1}{p+p''}$ and $\tau\equiv 0$.
\end{Thm}

Since the functions $p(\theta)$ and $t(\theta)$ uniquely determine a family of horizontal lines by ~{Proposition \ref{prop1}}, we say that the horizontal envelope $\gamma$ in Theorem \ref{mainthm1} is generated by the family of horizontal lines $(p(\theta),\theta, t(\theta))$.

\begin{remark}\label{pconst}
A geodesic in $\mathbb{H}_1$ is a horizontally regular curve with minimal length with respect to the Carnot-Carath\'{e}odory distance. For two given points in $\mathbb{H}_1$ one can find, by Chow's connectivity Theorem (\cite{Gr} p. 95), a horizontal curve joining these points. Note that when $p\equiv c$, a constant function, by \eqref{projpeq3}, \eqref{thmcondi}, and \eqref{maineq1}, the horizontal envelope $\gamma$ generated by the family of horizontal lines $\left\{(c, \theta, t(\theta)),\ \theta\in[0,2\pi]\right\}$ is a (helix) geodesic with radius $c$; the same result occurs if $t$ is also a constant function. In particular, when $p\equiv 0$, the horizontal envelope is the line segment contained in the $z$-axis.
\end{remark}

Besides, the reversed statement of Theorem \ref{mainthm1} also holds for horizontal curves with "jumping" ends.
\begin{Thm}\label{mainthm2}
Let $\gamma:\theta \in [0,2\pi]\mapsto (x(\theta),y(\theta),z(\theta))\in \mathbb{H}_1$ be a horizonal curve with finite length. Suppose $x(0)=x(2\pi)$, $y(0)=y(2\pi)$, and $z(0)\neq z(2\pi)$. Then the set of its tangent lines is uniquely determined by $p=p(\theta)\geq 0$ and $t=t(\theta)$ satisfying \eqref{thmcondi}.
\end{Thm}

We introduce a method to construct horizontal envelopes.
\begin{coro}\label{coro1}
Let $p_i=p_i(\theta)\geq 0, t_i=t_i(\theta)$ be $C^1$-functions defined on $[0,2\pi]$ satisfying \eqref{thmcondi} for $i=1,2$. Suppose $\gamma_i=(x_i,y_i,z_i)$ is a horizontal envelope generated by the family of horizontal lines $(p_i,\theta, t_i)$ for $i=1,2$. Denote by $p=p_1+p_2$ and $t=t_1+t_2$. The curve $\gamma=(x,y,z)$ is a horizontal envelope generated by a family of horizontal lines $(p(\theta),\theta, t(\theta))$ if and only if
\begin{align}\label{coro1condi}
p_1p_2=p'_1p'_2.
\end{align}
\end{coro}

By Corollary \ref{coro1} and Remark \ref{pconst}, we know that if at least one of $\gamma_1$ or $\gamma_2$ is a (helix) geodesic with nonzero radius, $\gamma$ can not be a horizontal envelope.

Now we seek the classification of functions $p_1$, $p_2$ satisfying the condition $p_1p_2=p'_1p'_2$ in Corollary \ref{coro1}. Actually, for any subinterval $[a,b]\subset [0,2\pi]$ such that $p_i\neq 0$ and $p'_i\neq 0$ for $i=1,2$, the condition \eqref{coro1condi} is equivalent to $p_1(\theta)=p_1(a)\exp\left(\int_a^{\theta}\frac{p_2(\alpha)}{p'_2(\alpha)}d\alpha\right)$ for any $\theta\in [a,b]$. In addition, if $p_i(\theta_i)=0$ for some $\theta_i\in [0, 2\pi]$, we may move the horizontal envelope $\gamma$ by a left translation such that $p_i>0$ in $[0,2\pi]$. Without loss of generality we may assume that $p_i>0$ on $[0,2\pi]$ and obtain the corollary of classification.
\begin{coro}\label{classification}
Let $p_i(\theta)>0$ ($i=1,2$) be the $C^2$-functions defined on $[0,2\pi]$ satisfying the condition $p'_1p'_2=p_1p_2$. Suppose $p'_i\neq 0$ and $p''_i\neq 0$ in some subinterval $[a,b]\subset [0,2\pi]$. We have the following results in $[a,b]$:
\begin{enumerate}
\item[(1)] If $p'_1>0, p'_2>0$, $p''_2<0$, then $p''_1>0$.
\item[(2)] If $p'_1>0, p'_2>0$, $p''_2>0$, then $p''_1>0$.
\item[(3)] If $p'_1<0, p'_2<0$, $p''_2<0$, then $p''_1>0$.
\item[(4)] If $p'_1<0, p'_2<0$, $p''_2>0$, then $p''_1<0$.
\end{enumerate}
\end{coro}

Finally, we emphasis that unlike the differential system \eqref{system1} mentioned in the first paragraph, the horizontal envelope in $\mathbb{H}_1$, in general, does not have the exact expression similar to the one in \eqref{envelope}. Indeed, an alternative expression of a horizontal line can be obtained by the intersection of two planes in $\mathbb{H}_1$
\begin{align}
F_1(x,y,z,\theta):&=\cos\theta x +\sin\theta y-p=0, \label{F1}\\
F_2(x,y,z,\theta):&=-p\sin\theta x + p\cos\theta y + z-t=0, \label{F2}
\end{align}where the set of points such that $F_1(x,y,z,\theta)=0$ is a vertical plane passing through the line $p=x\cos\theta+y\sin\theta$, and the set of $F_2(x,y,z,\theta)=0$ is the contact plane spanned by $\mathring{e}_1(Q')$ and $\mathring{e}_2(Q')$ through the point $Q'=(p\cos\theta, p\sin\theta, t)$ (see Fig. \ref{fig2}).
\begin{figure}[ht!]
   \includegraphics[width=1.05\textwidth]{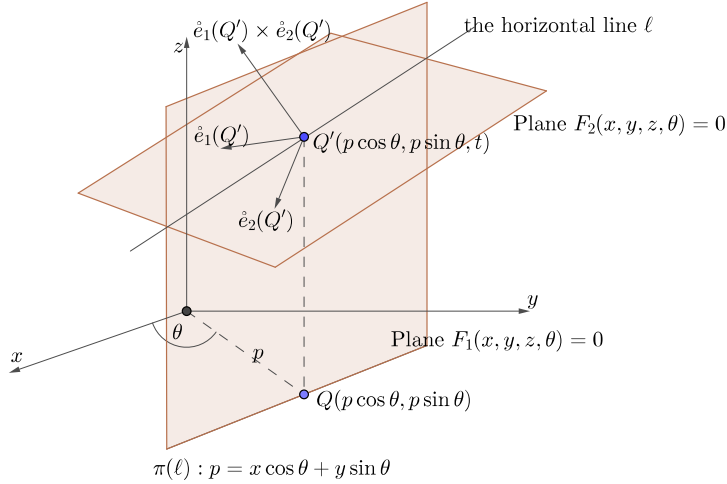}
   \centering \caption{Alternative expression of horizontal line $\ell$ in $\mathbb{H}_1$}
    \label{fig2}
\end{figure}
By taking the derivatives, respectively, of $F_1$ and $F_2$ with respect to $\theta$, we have
\begin{align}
\frac{\partial F_1}{\partial\theta}&=-x\sin\theta +y\cos\theta-p' =0 \label{F1d},\\
\frac{\partial F_2}{\partial\theta}&=(-p'\sin\theta-p\cos\theta)x+(p'\cos\theta-p\sin\theta)y-t'=0. \label{F2d}
\end{align}
Using \eqref{F1}, \eqref{F1d}, and substituting $p$, $p'$, into \eqref{F2d}, the condition $\frac{\partial F_2}{\partial\theta}=0$ is equivalent to \eqref{thmcondi}. Therefore, it may be only for seldom special cases that one can eliminate the parameter $\theta$ in \eqref{F1}, \eqref{F2}, and \eqref{F1d}, to find the exact expression for the horizontal envelope.

\section{Proofs of Theorems and Corollaries}
In this section we shall give the proofs of Theorem \ref{mainthm1} and \ref{mainthm2}. First we prove Proposition \ref{prop1} which plays the essential role in the notes.

\begin{proof}(Proposition \ref{prop1})
Although the proof has been shown in \cite{CHL} Proposition 8.2, for the self-contained reason we describe the proof here. Suppose $\pi(\ell)$ is the projection of horizontal line $\ell\in\mathbb{H}_1$ onto the $xy$-plane and the function $p=p(\theta)$ is the distance from the origin to $\pi(\ell)$ with angle $\theta$ from the positive direction of the $x$-axis (see Fig. \ref{envelope1}).
\begin{figure}[ht!]
   \includegraphics[width=0.8\textwidth]{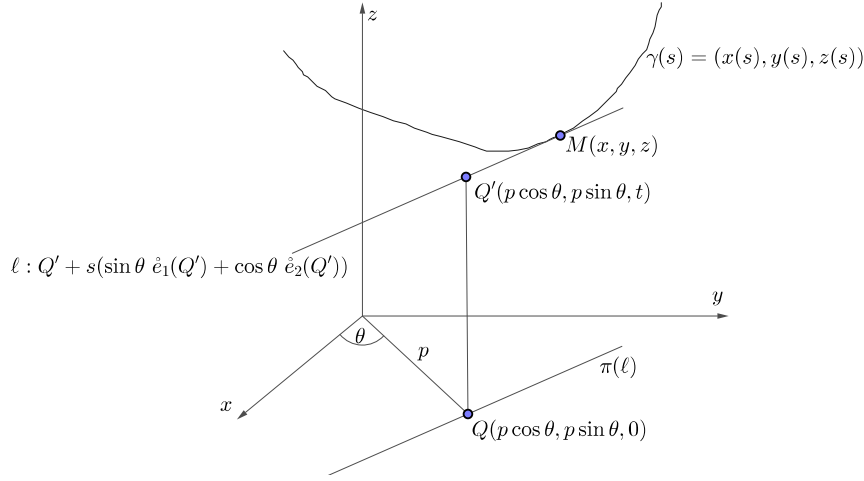}
   \centering \caption{The horizontal line $\ell$ in $\mathbb{H}_1$}
    \label{envelope1}
\end{figure}

Let the point $Q(p\cos\theta, p\sin\theta, 0)$ be the intersection of the line through the origin perpendicular to $\pi(\ell)$. We choose the unit directional vector $(-\sin\theta, \cos\theta,0)$ along $\pi(\ell)$, and so any point $(x,y,0)$ on $\pi(\ell)$ can be parametrized by arc-length $s$
\begin{align}\label{projpeq1}
\left\{
\begin{array}{rl}
x&=p\cos\theta - s\ \sin\theta, \\
y&=p\sin\theta + s\ \cos\theta.
\end{array}\right.
\end{align}
Denote the lift of the point $Q$ on $\ell$ by $Q'$. We may assume $Q'=(p\cos\theta, p\sin\theta, t)$ for some $t\in \mathbb{R}$. Since $\ell$ is horizontal, the tangent line of $\ell$ must be on the contact plane $\xi$. Thus, the parametric equations for $\ell$ can be represented by
\begin{align}\label{projpeq2}
\ell: (x,y,z)=(p\cos\theta, p\sin\theta, t) + s (A \mathring{e}_1(Q')+B \mathring{e}_2(Q')),
\end{align}
for some constants $A,B$ to be determined. Expand \eqref{projpeq2} by using the definitions of $\mathring{e}_1$ and $\mathring{e}_2$, and compare the coefficients with \eqref{projpeq1}, one gets $A=-\sin\theta$ and $B=\cos\theta$. Thus, the parametric equations for $\ell$ are obtained as shown in \eqref{projpeq3}. By \eqref{projpeq3}, the parameters $p,\theta$, and $t$, uniquely determine the horizontal line $\ell$.
\end{proof}

\begin{prop}\label{touch}
Given any horizontal curve $\gamma(s)=(x(s),y(s),z(s))$ parametrized by horizontal arc-length $s$. Suppose the horizontal line $\ell$ determined by $(p,\theta, t)$ intersects $\gamma$ at the unique point $M(x,y,z)$ on the contact plane $\xi_M$ (see Fig. \ref{envelope1}). If $p=p(\theta)$ is a function of $\theta$, then the intersection $M(x,y,z)$ can be uniquely represented by
\begin{align}\label{tngpt}
\left\{
\begin{array}{rl}
x&= p\cos\theta -p' \sin\theta, \\
y&= p\sin\theta +p' \cos\theta, \\
z&=t-p' p,
\end{array}
\right.
\end{align}where $p'$ denotes the partial derivative of the function $p$ with respect to $\theta$.
\end{prop}

\begin{proof}
We first solve the projection $(x,y,0)$ of the intersection $M$ onto the $xy$-plane in terms of $p$ and $\theta$, and then the $z$-component of $M$. By \eqref{projpeq1}, any point $(x,y)$ on the projection of $\ell$ satisfies
\begin{align}\label{projeq4}
p=x\cos\theta+y\sin\theta.
\end{align}
Since $p=p(\theta)$, take the derivative on both sides to get
\begin{align}\label{projeq5}
p'=-x\sin\theta + y\cos\theta.
\end{align}
Use \eqref{projeq4} and \eqref{projeq5} we obtain the first two components of the intersection $M$ on the $xy$-plane, namely,
\begin{align*}
x&=p\cos\theta -p' \sin\theta, \\
y&=p\sin\theta +p'\cos\theta.
\end{align*}
To determine the third component of $M$, $z$, by using \eqref{projpeq1}, \eqref{projeq5}, we have $s=y\cos\theta-x\sin\theta=p'$. Finally, \eqref{projpeq3} implies that $z=t-p'p$. The unique intersection point $M$ immediately implies the uniqueness of the expression \eqref{tngpt} and the result follows.
\end{proof}

Now we prove Theorem \ref{mainthm1}.
\begin{proof}(Theorem \ref{mainthm1})
According to the assumptions for the functions $p$ and $t$, the curve $\gamma(\theta)=(x(\theta),y(\theta),z(\theta))$ defined by \eqref{maineq1} is well-defined.
By Proposition 4.1 \cite{CHL}, since any horizontally regular curve can be reparametrized by its horizontal arc-length, it suffices to show that the curve $\gamma(\theta)$ is horizontal. Indeed, by Proposition \ref{touch}, a straight-forward calculation shows that
\begin{align*}
&z'-x'y+y'x\\
&=\left(t'-p''p-(p')^2 \right)-\left( p'\cos\theta-p\sin\theta-p''\sin\theta-p'\cos\theta\right)\left( p\sin\theta+p'\cos\theta\right)\\
&\hspace{1cm}+\left(p'\sin\theta+p\cos\theta+p''\cos\theta-p'\sin\theta\right)\left(p\cos\theta-p'\sin\theta\right)\\
&=t'-(p')^2+p^2.\\
\end{align*}Thus, $z'-x'y+y'x=0$ if and only if the functions $p$ and $t$ satisfy \eqref{thmcondi}. Therefore the curve defined by \eqref{maineq1} is horizontal by \eqref{horizontal}.

To derive the p-curvature for the curve $\gamma$, substitute \eqref{maineq1} into \eqref{pcurve} and the result follows. It is also clear that $\tau\equiv 0$ since $\gamma$ is horizontal.
\end{proof}

The horizontal length of the horizontal envelope can be represented by the function $p(\theta)$. Actually, by \eqref{maineq1} we have the horizontal length $$L(\gamma):=\int_0^{2\pi}\left[\left(x'(\theta)\right)^2+\left((y'(\theta)\right)^2\right]^{1/2} d\theta=\int_0^{2\pi}|p+p''| d\theta.$$
Compare the p-curvature $k$ in Theorem \ref{mainthm1} and the function on the right-hand side of the integral, we conclude that the length of the horizontal envelope $\gamma$ is the integral of the radius of curvature for the projection $\pi(\gamma)$
$$L(\gamma)=\int_0^{2\pi} \frac{1}{|k(\theta)|} d\theta.$$

Next we show Theorem \ref{mainthm2}.
\begin{proof}(Theorem \ref{mainthm2}) Suppose that the horizontal line $\ell$ represented by $(p,\theta, t)$ is tangent to $\gamma$ at $M(x,y,z)$. Since $M\in \ell$, by \eqref{projpeq3}, we can solve the point $Q'$ representing the horizontal line in terms of $x,y,z$. Let
\begin{align}\label{thm2xyz}
x&=p\cos\theta-s\ \sin\theta, \\
y&=p\sin\theta+s\ \cos\theta, \nonumber \\
z&=t-sp.\nonumber
\end{align}
One can solve
\begin{align}\label{thm2eq1}
s=-x\sin\theta + y\cos\theta.
\end{align}
Substitute $s$ into the third equation in \eqref{thm2xyz} to get
\begin{align}\label{thm2eq2}
t=z+(-x\sin\theta+y\cos\theta)p
\end{align}
The first two equations in \eqref{thm2xyz} imply that
\begin{align}\label{thm2eq3}
p=x\cos\theta+y\sin\theta,
\end{align}which means that the distance $p$ is a smooth function of $\theta$ defined on $[0,2\pi]$ if t$\ell$ intersects $\gamma$ at exactly one point. Similarly, \eqref{thm2eq2} also implies that $t$ is a smooth function of $\theta$. Finally, use \eqref{thm2eq1}\eqref{thm2eq2}\eqref{thm2eq3} it is easy to check that the horizontal line $\ell(p,\theta, t)$ satisfies the condition \eqref{thmcondi} for any $\theta$.
\end{proof}

\begin{proof}(Corollary \ref{coro1})
By Theorem \ref{mainthm1}, it suffices to show that the identity $t'=(p')^2-p^2$ holds. By assumption we have
\begin{align*}
t'-(p')^2+p^2&=t_1'+t_2'-(p_1'+p_2')^2+(p_1+p_2)^2\\
&= -2p_1'p_2'+2p_1p_2.
\end{align*}
Thus, $t'=(p')^2-p^2$ if and only if the identity $p_1'p_2'=p_1p_2$ holds.
\end{proof}

\begin{proof}(Corollary \ref{classification})
\begin{enumerate}
\item[(1)] Take the derivatives for $p'_1p'_2=p_1p_2$ with respect to $\theta$, we have $p''_1p_2'=p'_1p_2+p_1p'_2-p'_1p''_2$. Use the assumptions for the signs of $p_i$, $p'_i$, and $p''_2$, the left-hand side of the equation must be positive, and so $p''_1>0$.
\item[(2)] The condition $p'_1p'_2=p_1p_2$ is equivalent to $p_1(\theta)=p_1(a)\exp{\int_a^\theta  \frac{p_2}{p'_2}d\alpha}$. Take the derivative twice we have $p''_1(\theta)=p_1(a)\exp{\int_a^\theta \frac{p_2}{p'_2}d\alpha}\left( \frac{(p'_2)^2-p_2p''_2+(p_2)^2}{(p'_2)^2}\right)$. Thus, the sign of $p''_1$ is only determined by the numerator
    \begin{align}\label{quantity1}
    (p'_2)^2-p_2 p''_2+(p_2)^2.
    \end{align} We claim that $(p'_2)^2-p_2 p''_2>0$ under the assumptions, and so \eqref{quantity1} is positive. Indeed, if $p_2, p'_2>0$, then $0<\left(\log\frac{p_2}{p'_2}\right)'=\left(\log p_2\right)'-\left(\log p'_2\right)'=\frac{p'_2}{p_2}-\frac{p''_2}{p'_2}$, and the result follows.
\item[(3)] Use the similar method as (1).
\item[(4)] Similar to (2), it suffices to show that the sign of \eqref{quantity1} is negative. We shall show that $(p'_2)^2-\frac{p_2p''_2}{2}<0$ and $-\frac{p_2p''_2}{2}+(p_2)^2<0$, and combine both inequalities to have the result. On one hand, since $p_2>0$ and $p'_2<0$, we have $0<\left(\log\frac{p_2}{\sqrt{-p'_2}}\right)'=\left( \log p_2-\frac{1}{2}\log(-p'_2) \right)'=\frac{p'_2}{p_2}-\frac{p''_2}{2p'_2}$, namely, $(p'_2)^2<\frac{p_2p''_2}{2}$. On the other hand, since $p'_2<0$ and $-\frac{p'_2}{p_2}=\left(\log\frac{1}{p_2}\right)'>0>\frac{-2}{\left( \log (-p'_2)\right)'}=\frac{-2p'_2}{p''_2}$, one has $\frac{-p'_2}{p_2}>\frac{-2p'_2}{p''_2}$. Thus, multiply by $\frac{1}{-p'_2}$ on both sides to have $\frac{p''_2}{2}>p_2$ which implies $\frac{p''_2p_2}{2}>(p_2)^2$ and we complete the proof.
\end{enumerate}
\end{proof}

Finally we point out that the construction in Corollary \ref{coro1} can not obtain a closed horizontal envelope $\gamma$. Indeed, take the derivative with respect to $\theta$ in \eqref{maineq1} and use \eqref{thmcondi}, we have
$$
z'=t'-(p')^2-pp''=-p^2-pp'',
$$which is equivalent to
\begin{align}\label{zcondi}
z(\theta)= -\int_0^{\theta} p^2(\alpha)+p(\alpha)p''(\alpha)d\alpha + z(0) \ \text{, for any } \theta \in [0,2\pi].
\end{align}
If the curve was closed, say $z(0)=z(2\pi)$ and $p(0)=p(2\pi)$, by using Integration by Parts in \eqref{zcondi} one gets
\begin{align}\label{closedarea}
\int_0^{2\pi}p^2(\alpha)-\left(p'(\alpha)\right)^2d\alpha = 0.
\end{align}
However, according to Santal\'{o} \cite{San} (equation (1.8) in I.1.2) we know that
$$F=\frac{1}{2}\int_0^{2\pi}p^2(\alpha)-(p'(\alpha))^2 d\alpha,$$ where $F$ is the enclosed area of the projection $\pi(\gamma)$ of the curve $\gamma$ on the $xy$-plane. Therefore, \eqref{closedarea} is equivalent to that $\gamma$ must be a vertical line segment which contradicts with closeness of the curve.

\end{document}